\documentclass[smallextended]{amsart}
\usepackage{amssymb}
\usepackage{amsmath,latexsym}
\usepackage{graphicx, color}
\usepackage{pstricks}

\newcommand{\B}{\mathbf{b}}

\newcommand{\Zeta}{\mathbb{Z}}
\newcommand{\Z}{\mathbb{Z}}
\newcommand{\R}{\mathbb{R}}
\newcommand{\C}{\mathbb{C}}

\newcommand{\al}{\alpha}

\newcommand{\Ll}{\Lambda}
\newcommand{\Q}{\widetilde{Q}}
\newcommand{\de}{\delta}
\newtheorem{theorem}{Theorem}[section]
\newtheorem{lemma}[theorem]{Lemma}

\theoremstyle{definition}
\newtheorem{definition}[theorem]{Definition}
\newtheorem{example}[theorem]{Example}

\newtheorem{cor}[theorem]{Corollary}

\theoremstyle{remark}
\newtheorem{remark}[theorem]{Remark}

\begin{document}

\title{Approximation by crystal-refinable functions}

\author{Ursula Molter \and Mar\'ia del Carmen Moure \and Alejandro Quintero}


\address[U. Molter] {             Universidad de Buenos Aires,
Facultad de Ciencias Exactas y Naturales,
Departamento de Matem\'atica,
Buenos Aires, Argentina, CONICET-Universidad de Buenos Aires,
Instituto de Investigaciones Matem\'aticas Luis A. Santalo (IMAS).
Buenos Aires, Argentina}
             \email{umolter@dm.uba.ar}         
\address[M.C. Moure and A. Quintero]{Centro Marplatense de Investigaciones Matem\'aticas, Facultad de Ciencias Exactas y
  Naturales,  Universidad Nacional de Mar del Plata, Funes 3350, 7600
  Mar del Plata, Argentina.}
            \email{mcmoure@mdp.edu.ar}  \email{aquinter@mdp.edu.ar}


\maketitle

\begin{abstract}
Let $\Gamma$ be a crystal group in $\R^d$. A function $\varphi:\R^d\longrightarrow \C$ is said to be {\em crystal-refinable} (or $\Gamma-$refinable) if it is a linear combination of finitely many of the rescaled and translated functions $\varphi(\gamma^{-1}(ax))$, where the {\em translations} $\gamma$ are taken on a crystal group $\Gamma$, and $a$ is an expansive dilation matrix such that $a\Gamma a^{-1}\subset\Gamma.$ A $\Gamma-$refinable function $\varphi: \R^d \rightarrow \C$ satisfies a refinement equation $\varphi(x)=\sum_{\gamma\in\Gamma}d_\gamma \varphi(\gamma^{-1}(ax))$ with $d_\gamma \in \C$. Let $\mathcal S(\varphi)$ be the linear span of  $\{\varphi(\gamma^{-1}(x)): \gamma \in \Gamma\}$ and $\mathcal{S}^h=\{f(x/h):f\in\mathcal{S(\varphi)}\}$. One important property of $\mathcal S(\varphi)$ is, how well it approximates functions in $L^2(\R^d)$. This property is very closely related to the {\em crystal-accuracy} of $\mathcal S(\varphi)$, which is the highest degree $p$ such that all multivariate polynomials $q(x)$ of ${\rm degree}(q)<p$ are exactly reproduced from elements in $\mathcal S(\varphi)$. In this paper, we determine the accuracy $p$ from the coefficients $d_\gamma$. Moreover, we obtain from our conditions, a characterization of accuracy for a particular lattice refinable  vector function $F$, which simplifies the classical conditions.
\end{abstract}

\section{Introduction}

Crystal groups (Crystallographic groups or space groups), are groups of
isometries of $\mathbb{R}^{d}$ that generalize the notion of translations along
a lattice, allowing to move using different (rigid)
movements in $\mathbb{R}^{d}$ following a bounded pattern that is
repeated until it fills up space.
Precisely  (see \cite{Far81}):
\begin{definition}\label{defi-crista}
A \emph{crystal group} is a discrete subgroup $\Gamma \subset {\rm Isom}(\mathbb{%
R}^{d})$ such that ${\rm Isom}(\mathbb{R}^{d})/\Gamma$ is compact, where  ${\rm Isom}(\mathbb{R}^{d})$ is endowed with the pointwise convergence topology.

Or equivalently, one can define a \emph{crystal group} to be a discrete subgroup $\Gamma \subset {\rm Isom}(\mathbb{%
R}^{d})$ such that there exists a compact {\em fundamental domain} $P$ for $\Gamma$, i.e. there exists a bounded closed set $P$ such that
\[
\bigcup_{\gamma\in\Gamma}\gamma(P)=\R^d \mbox{ and } \gamma(P^\circ)\cap\gamma'(P^\circ)\neq\emptyset \mbox{ then } \gamma=\gamma',
\]
where $P^\circ$ is the interior of $P$.
\end{definition}

Note that the set of translations on a lattice is the simplest of the crystal groups.

It is known that $d-$dimensional crystal groups are intrinsically related to regular tessellations of $\R^d$, being $\Gamma=\{\tau_k:k\in\mathcal{L}\}$, the group of translations $(\tau_k(x)=x+k)$ on a lattice $\mathcal{L}$ the simplest example. From the beginning of wavelets it is clear that such tiling property of translations play a central role. The main idea in those systems, is to move a \emph{wave} through out the space, in such a way that every point is reached. Dilations of the wave are also required to obtain reproducing systems.

When we replace the translations in a lattice by movements on a crystal group, we have many more reproducing systems available without losing the conditions of moving at each scale under the action of a group (see Definition \ref{adm} ). If one just thinks of Haar wavelets, which are systems intrinsically associated with self-affine tiles we immediately  realize the universe of new systems that arises if we change the translations by transformations in a crystal group \cite{alfcar,Ge,GrMa}.

In this sense, \emph{crystallographic wavelets}, or \emph{crystal wavelets}, and its associated \emph{crystallographic mutiresolution analysis} are a natural generalization of classical wavelets and multiresolution analysis (\cite{Mer18}, Chapter 7). In these systems, a crystal group $\Gamma$ plays the role of translations in classical wavelets.

The group condition is not essential to building reproducing systems such as wavelets, but is desirable in order to allow the use of powerful mathematical tools \cite{Tay,Bag}. Further, if we want to ensure a regular movement throughout space (\emph{discrete and uniform}, see \cite{Vi}) by the action of a group of isometries, \emph{we can not have anything different than a crystal group.} As already mentioned, the group of translations on a lattice is the simplest of the crystal groups.

Accuracy has played an important role in both approximation theory and in wavelet theory. In approximation theory, it is closely related to the approximation properties of shift invariant spaces. In wavelet theory, one of the most successful and systematic ways of constructing smooth, compactly supported, orthonormal wavelet bases for $L^2(\R)$ is based on the factorization of a symbol which determines a scaling function \cite{D}. This factorization of the symbol is related to the accuracy of the scaling function. If the scaling function has accuracy $p$, then the corresponding wavelet will have $p$ zero moments. Hence accuracy is necessary for a refinable function to be smooth, although it is not sufficient. General results of accuracy can be found in \cite{CHM1,CHM2,CHM3,Jia} and references therein.

Our goal in this paper is to obtain necessary and/or sufficient conditions  for a crystal refinable  function $\varphi$ to have crystal accuracy $p$. In this direction, our first result establishes {\em necessary} conditions on $\mathcal S(\varphi)$  with $\varphi$ an arbitrary function (not necessarily refinable), to have crystal accuracy $p$. In the case that the function $\varphi$ is $\Gamma-$refinable, we will give {\em necessary and sufficient} conditions to ensure that $\varphi$ or $S(\varphi)$ has crystal-accuracy $p$. Using the results obtained for crystal refinable functions, accuracy conditions on the coefficients of the refinement equation for a special case of functions turn out to be  much simpler than in the general case (see Theorem~\ref{simple}). Finally  in Theorem~\ref{teo-accu-approx} we establish  Strang-Fix-type conditions adapted to our case.

Let us start recalling the necessary definitions.
\subsection{Crystal Groups}

For crystal groups (see Def.~\ref{defi-crista}), we have the fundamental theorem of Bieberbach \cite{B1}, \cite{Wo} which states the following:
\begin{theorem}[Bieberbach] Let $\Gamma$ be a crystal subgroup of ${\rm Isom}(\mathbb{R}^{d})$. Then
\begin{enumerate}
\item
$\Lambda=\Gamma \cap {\rm Trans}(\R^d)$  is a finitely generated abelian group of rank $d$
which spans ${\rm Trans}(\R^d)$, and
\item the linear parts of the symmetries $ad (\Gamma)$,  the {\em point group} of $\Gamma$, is finite, and satisfies  $ad (\Gamma) \cong \Gamma/\Lambda$.
\end{enumerate}
\end{theorem}
 (See also \cite{lomont},
IV-4). Here ${\rm Trans}(\R^d)$ stands for translations of $\R^d$.

We will denote the point group of $\Gamma$ by $G$.
and call $(\Gamma, G, \Lambda)$ a crystal triple.
\begin{remark} \label{2.3}
\
{\rm
\begin{itemize}
\item Note that the set $\Lambda$ is not empty by Bierberach's theorem \cite{B1} and consists of translations on the lattice $\Lambda$ which is isomorphic to $\Zeta^d$. By abuse of notation we will identify $\Lambda$ with the translations on $\Lambda$.

 We will denote by $L$ and $L^*$ the fundamental domains of the lattices $\Lambda$ and its dual, $\Lambda^*$ respectively. Here $\Lambda=R(\Zeta^d)$ with $R$ an invertible $d\times d$ matrix and hence $\mbox{$\Lambda^*=(R^*)^{-1}(\Zeta^d)$}$.
 \item The Point Group $G$ of $\Gamma$ is a finite subgroup of $\textbf{O}(d)$, the orthogonal group of $\mathbb{R}^{d}$, that preserves the lattice of translations, i.e. $G\Lambda = \Lambda$.
\end{itemize}
}
\end{remark}

General results on crystal groups, can be found for example in
\cite{gru}, \cite{Z}, \cite{martin}, \cite{B1}, \cite{B2}.

Note that
the simplest example of a crystal group is the group of
translations on a lattice $\Ll$, i.e. $\Gamma=\{\tau_k : \
k\in\Ll\}$, where $\tau_k(x)=x+k.$

One very important class of crystal groups, are the {\em splitting crystal groups}:
\begin{definition}\label{split}
$\Gamma$ is called a \emph{splitting crystal group} if it is the
semidirect product of the subgroups $\Lambda $ and $G$. In this case $\Gamma =\Lambda\rtimes G$,
 and for each
$\gamma, \widetilde{\gamma} \in \Gamma$, we have $\gamma \cdot \widetilde{\gamma}=( k+g \widetilde k, g\widetilde g)$, for
$\gamma=(k,g), \widetilde{\gamma} = (\widetilde k, \widetilde g)$ with $k, \widetilde k\ \in \Lambda$ and $g, \widetilde g \in G$ and $\gamma(x)=g(x) + k$.
\end{definition}

Every crystal group is naturally embedded in a splitting group, and very often arguments for general groups can be relatively easy reduced to the splitting case and then be proved for that simpler case.
This justifies, that from now on we will only consider splitting crystal groups.

For simplicity of notation, for each $\gamma\in\Gamma$ we will use the notation $\gamma=(k, g)$ in stead of $(\tau_k, g)$.

\begin{example}
Consider the vectors $u=(0,1)$ and
$v=(1,0)$ and let $S$ be the symmetry with respect to the
$X$-axis (i.e $S(x, y) = (x,-y)$).

Let $\Gamma$ be the group generated by $\{\tau_u,\tau_v, S\}$. Then
$\Lambda=\{\tau_{\ell} : \ell\in\Lambda\}$ where $\Lambda= \Z^2$ and $G=\{Id, S\}$. The fundamental domain $P$ is the
rectangle of vertices
$\{(0,0);(1,0);(0,1/2);(1,1/2)\}.$

%
\end{example}

\begin{definition}\label{adm}
 Let $\Gamma$ be a crystal group. We will say that $a \in \R^{d\times d}$ is a \emph{$\Gamma-$admissible matrix}, if $a$ is an expanding affine map and $a\Gamma a^{-1}\subset\Gamma.$
\end{definition}

It is easy to see that if $a$ is a $\Gamma-$admissible matrix, then $m=|\det a|$ is an integer. Therefore, the quotient group $\Gamma/a\Gamma a^{-1}$ is of order $m$.

A  function  $\varphi: \R^d \longrightarrow \C$ is \emph{$\Gamma$-refinable} with respect to $a$ and $\Gamma$ if it is a linear combination of the rescaled and `translated' functions $\varphi(\gamma^{-1}(ax))$, where the `translates' $\gamma\in\Gamma$ are movements on $\Gamma$. Precisely, $\varphi$ satisfies a \emph{refinement equation} or \emph{ dilation equation} of the form
\begin{equation}\label{ec_ref_single}
  \varphi(x)=\sum_{\gamma\in\Gamma'}d_\gamma \varphi(\gamma^{-1}(ax)),
\end{equation}
for some finite $\Gamma'\subset\Gamma$.

%
Refinable functions with respect to $a$ and $\Gamma$ are related to \emph{Crystal Wavelets} and \emph{Wavelets with composite dilations} \cite{alfcar}, \cite{GLLWW}, \cite{MQ16}.

In this paper we address the multidimensional case ($d\geq1$) with a $\Gamma-$admissible matrix $a$ for crystal-invariant spaces. We seek to determine one fundamental property of the space spanned by the $\Gamma-$refinable function $\varphi$ based on the coefficients $d_\gamma$: the property of providing good approximation in $L^2(\R^d)$. For the 1-dimensional case ($d=1$), and $\Gamma = \Z$, the approximation order is equivalent to the {\em accuracy} of the function  $\varphi$. For $d \geq 2$ unfortunately the equivalence is not true, however, accuracy is still {\em necessary} for providing good approximation (see \cite{Jia}). In section~\ref{approx-accuracy} we will elaborate on these relations for crystal-accuracy.

\begin{definition} Let $\varphi:\R^d \rightarrow \C$, the {\em crystal accuracy} of $\varphi$ is the largest integer $p$ such that {\em all} multivariate polynomials $q(x)=q(x_1, \dots, x_d)$ of ${\rm deg}(q)<p$ lie in the space that is the closure of all finite linear combinations of ${\Gamma-}$translates of the function $\varphi$,
\begin{equation}
S(\varphi)  = \overline{\rm span} \left\{\sum_{i=1}^k\ d_{\gamma_i} \ \varphi (\gamma_i(x))\; :\; d_{\gamma_i}\in\C \right\}. \label{SF}
\end{equation}
\end{definition}
As usual, equality of functions is interpreted as holding almost everywhere (a.e.). Note that in fact,  {\em accuracy} is a property of the space $\mathcal S(\varphi)$, but since the space is generated by $\Gamma-$translates of the function $\varphi$, we will talk in-distinctively about the accuracy of $\varphi$, or of $\mathcal S(\varphi)$. Just as a remark, we use this definition of $S(\varphi)$ for convenience of future calculations, but it is clear, that it also satisfies
\begin{equation*}
S(\varphi)  =   \overline{\rm span} \left\{\sum_{i=1}^k\ d_{\gamma_i} \ \varphi (\gamma_i^{-1}(x))\; :\; d_{\gamma_i}\in\C \right\}.
\end{equation*}

The results of this paper, for the most general case, of multidimensional {\em vector-valued} functions, can also be obtained in a similar way, however, the notation is even more complicated and the proofs are slightly more delicate. However the main ideas are already contained in the single function case $\varphi:\R^d\longrightarrow \C$ , and this is why we chose to present this case of a single function and in the appendix we state the general theorems without proof.


\section{Notation}

 We use the standard multi-index notation $x^\al=x_1^{\al_1}...x_d^{\al_d}$, where $x=(x_1,...,x_d)^T$ is in $\R^d$ and $\al=(\al_1,...,\al_d)$ with each $\al_i$ a nonnegative integer. The degree of $\al$ is $|\al|=\al_1+...+\al_d$. The number of multi-indices $\al$ of degree $s$ is
 $\displaystyle d_s=\left(                                                                                                                                   \begin{array}{c}
s+d-1 \\
d-1 \\
\end{array}
\right)$. We write
$\beta\leq\al$ if $\beta_i\leq\al_i$ for $i=1,...,d$.

Following the ideas in \cite{CHM1} for each integer $s\geq0$ we define the vector-valued function $X_{[s]}:\R^d\rightarrow\R^{d_s}$ by
$$X_{[s]}(x)=[x^\alpha]_{|\alpha|=s},\; x\in\R^s.$$

For our purposes we need define two special matrices, $a_{[s]}$ and $Q_{[s,t]}$ for integers $s,t\geq0$. Given a matrix $a$, we define the matrices $a_{[s]}$ and $Q_{[s,t]}$ by
\begin{eqnarray*}
X_{[s]}(ax)&=&a_{[s]}X_{[s]}(x),\\
X_{[s]}(x-y)&=&\sum_{t=0}^sQ_{[s,t]}(y)X_{[t]}(x).
\end{eqnarray*}
Note that $a_{[s]} \in \R^{d_s \times d_s}$ and $Q_{[s,t]} \in \R^{d_s \times d_t}$.

These matrices have two properties that will be of great importance.
\begin{lemma}\label{lema4}
Let $a \in \R^{d\times d}$ be a matrix, and $\Lambda$ be the lattice associated to the crystal group $\Gamma$ (see Remark~\ref{2.3}). Then:
\begin{enumerate}
  \item If $a$ is an expansive matrix then $a_{[s]}$ is an expansive matrix for each $s\geq0$.
  \item If $a$ is an invertible matrix then $Q_{[s,t]}(az)=a_{[s]}Q_{[s,t]}(z)(a^{-1})_{[t]}.$
\end{enumerate}
\end{lemma}
The proof of the previous lemma as well as the explicit form and properties of these matrices can be seen in \cite{CHM1}.

From the matrices $a_{[s]}$ and $Q_{[s,t]}$ in order to obtain \\
$X_{[s]} \left(\gamma^{-1}(x)\right)= \sum_{t=0}^s \Q_{[s,t]}(\gamma) X_{[t]}(x)$ we give the following definition.

\begin{definition}\label{q-tilde}
 Let $(\Gamma,G,\Lambda)$ be a splitting crystal triple. Let $\gamma\in\Gamma$, $\gamma=(l, g)$ then we define the matrices $\widetilde{Q}_{[s,t]}$ by
 $$\widetilde{Q}_{[s,t]}(\gamma)=(-1)^sg^{-1}_{[s]}Q_{[s,t]}(l),$$
 where  $g^{-1}_{[s]}$ is the matrix that satisfies $X_{[s]}(g^{-1}(x))=g^{-1}_{[s]}X_{[s]}(x)$. In the case that $\gamma = (l, {\rm Id})$ we will 
 write $\Q_{[s,t]} (\gamma) = \Q_{[s,t]}(l) = Q_{[s,t]}(l)$.
\end{definition}

\begin{lemma}\label{lem}
Let $(\Gamma,G,\Lambda)$ be a splitting crystal triple, and $a$ an invertible matrix such that $a\Gamma a^{-1}\subset\Gamma$. We then have:
  \begin{enumerate}
    \item $\Q_{[s,s]}(l,{\rm Id})={\rm Id}, l \in \Lambda.$
    \item $\Q_{[s,0]}(\gamma)=g^{-1}_{[s]}X_{[s]}(l)$ for each $\gamma=(l,g)\in\Gamma$.
    \item $\Q_{[s,t]}(\gamma_1\gamma_2)=\sum_{u=t}^s\Q_{[s,u]}(\gamma_2)\Q_{[u,t]}(\gamma_1)$.
    \item $\Q_{[s,t]}(a\gamma a^{-1})=a_{[s]}\Q_{[s,t]}(\gamma)a^{-1}_{[t]}$. \label{i4}
 \item Let $b_t\in\C^{d_t\times r}$ be given matrices,   for $0\leq t\leq s$. If\  $\sum_{t=0}^s\Q_{[s,t]}(al)b_t=0$ for each $l\in\Lambda$, then $b_t=0$ for $0\leq t\leq s$.  \label{5}
  \end{enumerate}
\end{lemma}
The proof the previous lemma is immediate from Lemma \ref{lema4} and Lemmas 4.1 and 4.7 of \cite{CHM1}.

Given a collection $\{v_\al \in\C :0\leq|\al|<p\},$
 we shall associate special matrices and functions, which play an important role in our analysis of accuracy.

We group the numbers $v_\al$ by degree to form column vectors $v_{[s]} \in \C^{d_s}$ , i.e.
\begin{equation}\label{v}
  v_{[s]}=[v_\al]_{|\al|=s}, \ \ 0\leq s<p.
\end{equation}
Note that, when $|\al|=0$ then $v_{[0]}=[v_0]_0=v_0$.

We define the matrices $y_{[s]}(\gamma)$ by
\begin{equation}\label{y-s}
y_{[s]}(\gamma)=\sum_{t=0}^s\Q_{[s,t]}(\gamma)v_{[t]} = g^{-1}_{[s]}\sum_{t=0}^sQ_{[s,t]}(l)v_{[t]},\end{equation}
where $\gamma=(l,g)$ and $g^{-1}_{[s]}$ is as before the matrix that satisfies $X_{[s]}(g^{-1}(x))=g^{-1}_{[s]}X_{[s]}(x)$.

Finally, we define the infinite row vector
\begin{equation} \label{Y-s}Y_{[s]}=(y_{[s]}(\gamma))_{\gamma\in\Gamma}.\end{equation}

The functions $y_{[s]}$ have the following properties.

\begin{lemma}\label{cor:4.4}
  Let $\{v_\al\in\C:0\leq|\al|<p\}$ be given and let   $y_{[s]}$ be the functions given by (\ref{y-s}). Let $\gamma_1$ and $\gamma_2$ in $\Gamma$, then
  $$y_{[s]}(\gamma_1\gamma_2)=\sum_{t=0}^s\Q_{[s,t]}(\gamma_2)y_{[t]}(\gamma_1).$$
\end{lemma}
\begin{proof}
For the proof we use Lemmas 4.1, 4.2 and 4.3 of \cite{CHM1}. By definition
\begin{eqnarray*}
y_{[s]}(\gamma _{1}\gamma _{2})&=&\sum_{t=0}^s\Q_{[s,t]}(\gamma_1\gamma_2)v_{[t]} = \sum\limits_{t=0}^{s}\sum\limits_{u=t}^{s}\Q_{[s,u]}(\gamma_{2})\Q_{[u,t]}(\gamma_1)v_{[t]}
\\
&=&\sum\limits_{u=0}^{s}\Q_{[s,u]}(\gamma_{2})\sum%
\limits_{t=0}^{u}\Q_{[u,t]}(\gamma_{1})v_{[t]}  = \sum\limits_{u=0}^{s}\Q_{[s,u]}(\gamma_{2})y_{[u]}(\gamma_{1}).
\end{eqnarray*}
\end{proof}
\begin{remark}
Note that   if $\gamma_2=(l_2, {\rm Id})=\tau_{l_2}$, then the previous equality yields
  $$y_{[s]}(\gamma_1\tau_{l_2})=\sum_{t=0}^s\Q_{[s,t]}(l_2)y_{[t]}(\gamma_1).$$
  \end{remark}

We will say that the translates of the function $\varphi$ along $\Gamma$ are \emph{$\Gamma-$independent} if for every choice of scalars $b_\gamma\in\C$,
$$\sum_{\gamma\in\Gamma}b_\gamma \varphi(\gamma x)=0 \; \mbox{ if and only if, }\; b_\gamma=0\;\mbox{ for every }\gamma.$$
Equivalently, for every choice of  an infinite row vector $\B=(b_\gamma)_{\gamma\in\Gamma}$,
$$\B\Phi(x)=0\; \mbox{ if and only if, }\; \B=0.$$

Here $\Phi(x)$ is the infinite column vector  with entries $\varphi(\gamma (x)),$ i.e.%
\begin{equation}\label{F}
\Phi(x)=\left[ \varphi(\gamma (x))\right] _{\gamma \in \Gamma }.
\end{equation}

\section{Characterization of Accuracy}

\subsection{Necessary conditions for arbitrary functions.}

In this section, we will present necessary conditions for an arbitrary (not necessarily $\Gamma-$refinable) function $f:\R^d \longrightarrow \C$ with $\Gamma-$independent translates, to have accuracy $p$.

\begin{theorem}\label{thm3}
Assume that $f:\mathbb{R}^{d}\rightarrow\mathbb{C}$ is
compactly supported, and that translates of $f$ are
$\Gamma-$independent. If $f$ has accuracy $p$ then there exists a
collection $\{v_{\alpha }\in\mathbb{C}: 0\leq
\left\vert \alpha \right\vert <p\}$ of row vectors such that
\begin{description}
\item[i)] $v_{0}\neq 0.$
\item[ii)] $X_{[s]}(x)=\sum\limits_{\gamma \in \Gamma }y_{[s]}(\gamma
)f(\gamma (x))=Y_{[s]}F(x)$ for $0\leq s<p,$ and $F(x)$ is as defined in (\ref{F})\\
where $Y_{[s]}=\left(
y_{[s]}(\gamma)\right) _{\gamma \in \Gamma } = \left(\sum_{t=0}^s\Q_{[s,t]}(\gamma)v_{[t]}\right) _{\gamma \in \Gamma }$ (as in (\ref{y-s}) and (\ref{Y-s})).
\end{description}
\end{theorem}

\begin{proof}
Since $f$ has accuracy $p$, there exist coefficients $w_{\alpha
,\gamma }\in \mathbb{C}$ such that every polynomial $x^{\alpha }$ of degree $\alpha,$ $%
0\leq \left\vert \alpha \right\vert <p$ can be written as a finite linear combination of $\Gamma-$translates of $f$,
$$
x^{\alpha }=\sum\limits_{\gamma \in \Gamma }w_{\alpha ,\gamma
}f(\gamma(x)) \mbox{ a.e.}$$
For each $\gamma \in \Gamma $, group the $w_{\alpha
,\gamma }$ by degree to form the column vectors
$$ w_{[s]}(\gamma
)=[w_{\alpha ,\gamma }]_{\left\vert \alpha \right\vert =s}.
$$
For each $\sigma \in \Gamma $ define the infinite row vector
$$W_{[s]}(\sigma )=\left( w_{[s]}(\gamma \sigma )\right) _{\gamma
\in \Gamma }.$$
 Next, let $v_{\alpha }=w_{\alpha,I}$
(where $I={\rm Id}$ is the identity of $\Gamma $) and recall the definitions of the vectors
$v_{[s]}$ and the matrices $y_{[s]}$ from (\ref{v}) and (\ref{y-s}).
Then we have
for $0\leq s<p,$ that
\begin{eqnarray*}
X_{[s]}(x) &=&[x^{\alpha }]_{\left\vert \alpha \right\vert =s} =\left[ \sum\limits_{\gamma \in \Gamma }w_{\alpha ,\gamma }f(\gamma(x))%
\right] _{\left\vert \alpha \right\vert =s} \\
&=&\sum\limits_{\gamma \in \Gamma }w_{[s]}(\gamma )f(\gamma(x))  = W_{[s]}(I)F(x).
\end{eqnarray*}
Now for each $\sigma =(\ell,g)\in \Gamma $ with $\sigma^{-1}=(-g^{-1}\ell, g^{-1})$
\begin{eqnarray*}
W_{[s]}(\sigma )F(x) &=& X_{[s]}(\sigma^{-1}(x)) = X_{[s]}(g^{-1}(x -\ell)) = \sum_{t=0}^s\Q_{[s,t]}(\sigma)X_{[t]}(x)\\
&=&\left(
\sum\limits_{t=0}^{s}\Q_{[s,t]}(\sigma)W_{[t]}(I)\right)
F(x).
\end{eqnarray*}
Taking into account our assumption that translates of $f$ are $\Gamma-$independent,
this implies that $W_{[s]}(\sigma
)=\sum\limits_{t=0}^{s}\Q_{[s,t]}(\sigma)W_{[t]}(I),
$
and therefore for each $\gamma \in \Gamma,
w_{[s]}(\gamma  \sigma ) =
\sum\limits_{t=0}^{s}\Q_{[s,t]}(\sigma)w_{[t]}(\gamma )$.
In particular, for $\gamma =I$ we obtain $w_{[s]}(\sigma)=y_{[s]}(\sigma )$.

Thus
$$ X_{[s]}(x)=\sum\limits_{\gamma \in \Gamma }y_{[s]}(\gamma
)f(\gamma (x))=Y_{[s]}F(x).$$

For $s=0$, since $y_{[0]}(\gamma )=v_{0}$ for
every $\gamma \in \Gamma $ we have
$$
1=x^{0}=X_{[0]}(x)=\sum\limits_{\gamma \in \Gamma }y_{[0]}(\gamma
)f(\gamma (x))=v_{0}\sum\limits_{\gamma \in \Gamma }f(\gamma (x)),
$$
and hence $v_{0}\neq 0.$
\end{proof}

\subsection{Accuracy for $\Gamma-$refinable functions.}

In this section we will obtain necessary and/or sufficient conditions for a $\Gamma-$refinable function to have accuracy $p$.

First, we rewrite the refinement equation  (\ref{ec_ref_single}) in matrix form.

Let $(\Gamma,G,\Lambda)$ be a splitting crystal triple and $a \in \R^{d\times d}$ a $\Gamma-$admissible matrix. Remember that a function $f:\R^d \longrightarrow \C$ is $\Gamma-$refinable if it satisfies
$$ f(x)=\sum_{\gamma\in\Gamma}d_\gamma f(\gamma^{-1}(ax)), \ \mbox{with} \ d_{\gamma} \in \C.  $$

We consider as before  (\ref{F}), $F(x)$ to be the infinite column vector $F(x)=[f(\gamma(x))]_{\gamma\in\Gamma}$.
Note that if $f$ has compact support, for a given $x$, only finitely many entries $f(\gamma(x))$ of $F(x)$ are non zero.

\begin{lemma}
  Let $f:\R^d\rightarrow\C$, $a \in \R^{d\times d}$ a $\Gamma-$admissible matrix and $F$ the function defined by $F(x)=\left[ f(\gamma  (x))\right] _{\gamma \in \Gamma }$ (see  (\ref{F})). Then, the function $f$ is $\Gamma-$refinable if and only if $LF(ax)=F(x)$ a.e., where $L$ is the $\Gamma \times \Gamma$ matrix given by $L=\left[ d_{a\gamma a^{-1}\sigma^{-1} }\right]
_{\gamma ,\sigma \in \Gamma }$,  where $d_\gamma$ are the coefficients of the refinement equation.
\end{lemma}
The proof of this result, is a consequence of the definition of the function $F$ and the matrix $L$.

The following result characterizes the accuracy of $\Gamma-$refinable functions.

\begin{theorem}\label{teo-3}
Assume that $f:\mathbb{R}^{d}\rightarrow\mathbb{C}$ is integrable, compactly supported and satisfies
the refinement equation (\ref{ec_ref_single}).  Consider the following statements

\begin{itemize}
\item[I)] $f$ has accuracy $p$.

\item[II)] There exist a collection of complex numbers $\{v_{\alpha }\in\mathbb{C} : 0\leq \left\vert \alpha \right\vert <p\}$
such that

\begin{itemize}
\item[(i)] $v_{0}\hat{f}(0)\neq 0$ and

\item[(ii)] $Y_{[s]}=a_{[s]}Y_{[s]}L$ for $0\leq s<p$
where $Y_{[s]}=(y_{[s]}(\gamma ))_{\gamma \in \Gamma }$ as in (\ref{y-s}) and (\ref{Y-s}).
\end{itemize}
\end{itemize}

Then we have the following:

\begin{itemize}
\item[a)] If the translates of $f$ along $\Gamma$ are independent, then \emph{(I)}
implies \emph{(II)}.

\item[b)] \emph{(II)} implies \emph{(I)}. In this case, if we
scale all the vectors $v_{\alpha }$ by
$C=(v_{0}\hat{f}(0))^{-1}\left\vert P\right\vert $ then
$$
X_{[s]}(x)=\sum\limits_{\gamma \in \Gamma }y_{[s]}(\gamma
)f(\gamma (x))=Y_{[s]}F(x)\mbox{, }0\leq s<p.
$$
\end{itemize}
\end{theorem}

\begin{proof}
\

\begin{itemize}
\item[a)] Since $f$ has accuracy $p$ and translates of $f$ along
$\Gamma$ are independent, by \emph{Theorem \ref{thm3}} there exists a collection of coefficients $\{v_{\alpha }\in\mathbb{C}: 0\leq \left\vert \alpha \right\vert <p\}$ such that
$$
X_{[s]}(x)=Y_{[s]}F(x)\mbox{ }0\leq s<p,
$$
with $y_{[s]}$ and $Y_{[s]}$ given by (\ref{y-s}) and (\ref{Y-s}) respectively, and $v_0\neq0$.

Further, if $P$ is a fundamental domain of $\Gamma $ then
$$
v_{0}\hat{f}(0)=v_{0}\int\limits_{\mathbb{R}^{d}}f(x)dx=v_{0}\sum\limits_{\gamma
\in \Gamma }\int\limits_{P}f(\gamma
(x))dx=\int\limits_{P}1dx=\left\vert P\right\vert \neq 0,
$$
which proves (i).

To prove (ii), using the refinement equation $F(x)=LF(ax)$ and the definition of
$a_{[s]}$ we see that
$$
Y_{[s]}F(ax) =X_{[s]}(ax) =a_{[s]}X_{[s]}(x) =a_{[s]}Y_{[s]}F(x) =a_{[s]}Y_{[s]}LF(ax),
$$%
and since $f$ has independent $\Gamma-$translates, this implies that
$Y_{[s]}=a_{[s]}Y_{[s]}L$ for $0\leq s<p$ which completes the proof of a).

\item[b)] For each $0\leq s<p$, define the vector-valued function
$G_{[s]}:\mathbb{R}^{d}\rightarrow\mathbb{C}^{d_s},$
by $$G_{[s]}(x)=\sum\limits_{\gamma \in \Gamma }y_{[s]}(\gamma
)f(\gamma (x))=Y_{[s]}F(x).$$

Note that for each fixed $x$, only finitely many terms in the
sum defining $G_{[s]}(x)$ are nonzero.

Using the equation $Y_{[s]}=a_{[s]}Y_{[s]}L$ and the refinement equation $LF(ax)=F(x)$,
we have
\begin{equation}
G_{[s]}(ax) =Y_{[s]}F(ax)  = a_{[s]}Y_{[s]}LF(ax)  = a_{[s]}Y_{[s]}F(x) =a_{[s]}G_{[s]}(x).  \label{4.1}
\end{equation}
Since $X_{[s]}(ax)=a_{[s]}X_{[s]}(x),$ we see that $G_{[s]}(x)$ and
$X_{[s]}(x)$ behave identically under dilation by $a$. We will show
that  if we take $C=(v_{0}\hat{f} (0))\left\vert
P\right\vert ^{-1}$, then $
G_{[s]}(x)=CX_{[s]}(x)$ for $0\leq s<p$. So $G_{[s]}$ coincides with $X_{[s]}, 0 \leq s < p$ - up to a constant that does not depend on $s$.

The quotient $\R^d/\Lambda$ is a compact abelian group, equipped with the normalized Haar measure. Let $\Pi:\R^n\rightarrow\R^n/\Lambda$, be the canonical projection onto the quotient.

The map $\boldsymbol{\tau}:=\Pi a\Pi^{-1}:\R^d/\Lambda\rightarrow\R^d/\Lambda$ is a well defined, measure preserving, continuous and surjective endomorphism of the group $\R^d/\Lambda$.

The group of the characters of $\R^d/\Lambda$ is given by
$$(\R^d/\Lambda)^\wedge =\{\gamma_\lambda:\R^d/\Lambda\rightarrow S; \ \gamma_{\lambda}(x)=e^{2\pi i\langle x,\lambda\rangle}, \mbox{ with }\lambda\in\Lambda^*\}.$$
If $\gamma_\lambda\circ\boldsymbol{\tau}^n=\gamma_\lambda$ for some $n\in\mathbb{N}$, then $e^{2\pi i\langle \tau^n x,\lambda\rangle}=e^{2\pi i\langle x,\lambda\rangle}$ for all $x\in\R^d/\Lambda$ or equivalently $e^{2\pi i\langle x,(a^n)^t\lambda\rangle}=e^{2\pi i\langle x,\lambda\rangle}$ for all $x\in\R^d/\Lambda$. Therefore $(a^n)^t\lambda=\lambda$ and since $a^n$ is expansive, $\lambda=0$. Hence $\gamma_\lambda\circ\boldsymbol{\tau}^n=\gamma_\lambda$ if and only if $\gamma_\lambda=1$. Therefore, by Theorem 1.10 of \cite{Wal82}, the map $\boldsymbol{\tau}$ is ergodic.

We now proceed by induction on $s$ to show that $G_{[s]}(x)=CX_{[s]}(x)$ for $
0\leq s<p$ with $C$ independent of $s$.

For $s=0$ $G_{[0]}(x)$ is scalar-valued.
Since $a_{[0]}$ is the constant $1$, Eq.~(\ref{4.1}) states that
$G_{[0]}(ax)=G_{[0]}(x)$. Further, $y_{[0]}(\gamma)=v_0$ for every
$\gamma\in\Gamma$, so $G_{[0]}(x)=\sum\limits_{\gamma
\in\Gamma}v_0f(\gamma (x))$. Therefore, for each $\ell\in \Lambda$ we have
$$G_{[0]}(x-\ell)=\sum\limits_{\gamma \in\Gamma}v_0f(\gamma
(x-\ell))=\sum\limits_{\gamma \in\Gamma}v_0f(\gamma
\tau_{-\ell}(x))=\sum\limits_{\gamma \in\Gamma}v_0f(\gamma (x)).$$
Thus $G_{[0]}(x)$ satisfies
$$G_{[0]}(ax)=G_{[0]}(x) \mbox{ and }G_{[0]}(x-\ell)=G_{[0]}(x)\mbox{ for all } \ell\in \Lambda.$$
Hence $G_{[0]}(\tau(x))=G_{[0]}(x)$ for each $x\in \R^d/\Lambda$. Since $\boldsymbol{\tau}$
is ergodic, it follows that $G_{[0]}$ is constant a.e on $L$, where $L$ is the fundamental domain of $\Lambda$ (Theorem 1.6 of \cite{Wal82}). By periodicity, we therefore have $G_{[0]}(x)=C$ a.e.
on $\mathbb{R}^d$. Explicitly,
$$C|P| = \int_PG_{[0]}(x)dx = v_0\sum\limits_{\gamma\in\Gamma}\int_P f(\gamma(x)) = v_0\int_{\mathbb{R}^d}f(x)dx = v_0\hat{f}(0)\neq0.$$
In particular $C=\left(v_0(\widehat{f})(0)\right)|P|^{-1}\neq0$. Suppose
now, inductively, that $G_{[t]}(x)=CX_{[t]}(x)$ a.e.~for $0\leq t<s$.
Then we have
\begin{eqnarray*}
G_{[s]}(x-\ell)&=&Y_{[s]}F(x-\ell) =\sum\limits_{\gamma\in\Gamma}y_{[s]}(\gamma)f(\gamma\tau_{-\ell}(x)) =\sum\limits_{\sigma\in\Gamma}y_{[s]}(\sigma\tau_{\ell})f(\sigma(x))\\
&=&\sum\limits_{t=0}^sQ_{[s,t]}(\ell)\sum\limits_{\sigma\in\Gamma}y_{[t]}(\sigma)f(\sigma(x)) \mbox{ by Lemma \ref{cor:4.4}}.
\end{eqnarray*}
This yields
\begin{eqnarray*}
G_{[s]}(x-\ell)&=&
\sum\limits_{t=0}^sQ_{[s,t]}(\ell)Y_{[t]}F(x) = \sum\limits_{t=0}^sQ_{[s,t]}(\ell)G_{[t]}\\
&=&Q_{[s,s]}(\ell)G_{[s]}(x)+\sum\limits_{t=0}^{s-1}Q_{[s,t]}(\ell)G_{[t]}(x).
\end{eqnarray*}
Using the inductive hypothesis, we have
\begin{eqnarray*}
G_{[s]}(x-\ell)&=& Q_{[s,s]}(\ell)G_{[s]}(x)+C\sum\limits_{t=0}^{s}Q_{[s,t]}(\ell)X_{[t]}(x)-CQ_{[s,s]}(\ell)X_{[s]}(x)\\
&=&G_{[s]}(x)+CX_{[s]}(x-\ell)-CX_{[s]}(x) \mbox{ by definiton of
}Q_{[s,t]}.
\end{eqnarray*}

Therefore, if we define $H_{[s]}(x)=G_{[s]}(x)-CX_{[s]}(x)$ then
$$H_{[s]}(ax)=a_{[s]}H_{[s]}(x) \mbox{ and } H_{[s]}(x-\ell)=H_{[s]}(x),
\mbox{ for }\ell\in \Lambda.$$
This implies that
$$H_{[s]}(\boldsymbol{\tau}(x))=a_{[s]}H_{[s]}(x).$$
Let now $E\subset L$ be a set of positive measure on which $H_{[s]}$ is
bounded, say $\|H_{[s]}(x)\|\leq M$ for $x\in E$, where $\|\cdot\|$
is any fixed norm on $\mathbb{C}^{d_s}$. Since $\boldsymbol{\tau}$ is ergodic, by  Birkhoff's Ergodic Theorem (see \cite{birk}) for almost every $x\in L$,
\begin{equation}\label{4}
\lim_{n \rightarrow\infty}\frac{\#\{0<k\leq n : \boldsymbol{\tau}^k(x)\in
E\}}{n}=|E|>0.
\end{equation}
Let $x\in {L}$ be such that (\ref{4}) holds. Then there exists an increasing sequence $\{n_j\}_{j=1}^\infty$ of positive integers such that $\tau^{n_j}(x)\in E$ for each $j$.  Hence
$$M\geq\|H_{[s]}(\tau^{n_j}(x))\|=\|a_{[s]}^{n_j}H_{[s]}(x)\|.$$
But since $a_{[s]}$ is expansive $\|a_{[s]}^{n_j}H_{[s]}(x)\| \longrightarrow \infty$  if $H_{[s]}(x)\neq 0$. Therefore we must have $H_{[s]}(x)=0$ a.e. on $\mathfrak{L}$. Since $H_{[s]}$ is $\Lambda$-periodic, it must therefore vanish a.e. on $\mathbb{R}^d$. Hence $G_{[s]}=CX_{[s]}$ a.e., which completes the proof.
\end{itemize}
\end{proof}

Since the conditions for accuracy given in the previous theorem are rather difficult to check, we follow \cite{CHM1}  to give several equivalent formulations for condition \emph{(ii)} in statement \emph{(II)}.

\begin{theorem}\label{thm} Assume that $f:\mathbb{R}^{d}\rightarrow\mathbb{C}$ is integrable, compactly supported and satisfies
the refinement equation (\ref{ec_ref_single}).
Let $m=\left\vert \det a\right\vert ,$ and let $\gamma_{1}, \dots,\gamma _{m}$
be a full set of digits of the left cosets of $\Gamma $. Here,  the left cosets $\Gamma_i$ are $\Gamma _{i}=\gamma _{i}a\Gamma a^{-1}.$

Given a collection $\{v_{\alpha }\in\mathbb{C}: 0\leq \left\vert \alpha \right\vert <p\}$,
$y_{[s]}(\gamma)=\sum\limits_{t=0}^{s}\Q_{[s,t]}(\gamma)v_{[t]}$ and  $Y_{[s]}=(y_{[s]}(\gamma ))_{\gamma \in \Gamma }.$

If $v_{0}\neq 0,$ then the following statements are equivalent:

\begin{itemize}
\item[a)] $Y_{[p-1]}=a_{[p-1]}Y_{[p-1]}L$.
Equivalently, \\$%
y_{[p-1]}(\sigma )=a_{[p-1]}\sum\limits_{\gamma \in \Gamma
}y_{[p-1]}(\gamma )d_{a\gamma a^{-1}\sigma ^{-1}}$ for $\sigma
\in \Gamma. $

\item[b)] $Y_{[s]}=a_{[s]}Y_{[s]}L$ for $0\leq s<p$. Equivalently, \\
$%
y_{[s]}(\sigma )=a_{[s]}\sum\limits_{\gamma \in \Gamma
}y_{[s]}(\gamma )d_{a\gamma a^{-1}\sigma ^{-1}}$ for $\sigma
\in \Gamma. $

\item[c)] $y_{[s]}(\gamma _{i})=a_{[s]}\sum\limits_{\gamma \in \Gamma
}y_{[s]}(\gamma )d_{a\gamma a^{-1}\gamma _{i}^{-1}}$ for $0\leq s<p$ and $i=1, \dots, m.$

\item[d)] $v_{[s]}=\sum\limits_{\gamma \in \Gamma
_{i}}\sum\limits_{t=0}^{s}\Q_{[s,t]}(\gamma^{-1})a_{[t]}v_{[t]}d_{\gamma
^{-1}}$ for $0\leq s<p$ and $i=1, \dots, m.$
\end{itemize}
\end{theorem}
Note that by this theorem, if one wants to check for accuracy $p$, one does not need to check {\bf all} conditions $0\leq s < p$, but it is enough to check it for $s=p-1$.

\begin{proof}

$b)\Rightarrow a)$ and $b)\Rightarrow c)$ are trivial. So we will prove $a)\Rightarrow b)$, $c)\Rightarrow b)$ and $c)\Leftrightarrow d)$.

$a)\Rightarrow b)$\\
Assume that (a) holds, we consider for $j\in \Lambda,$ and $\sigma\in\Gamma$
\begin{equation}
\sum\limits_{s=0}^{p-1}\Q_{[p-1,s]}(aj)\left(
a_{[s]}\sum\limits_{\gamma \in \Gamma }y_{[s]}(\gamma )d_{a\gamma
a^{-1}\sigma ^{-1}}\right).   \label{a}
\end{equation}
Then by Lemmas~\ref{lema4} and \ref{cor:4.4} we have that
\begin{eqnarray*}
(\ref{a})
&=&\sum\limits_{s=0}^{p-1}a_{[p-1]}\Q_{[p-1,s]}(j)a_{[s]}^{-1}a_{[s]}\sum%
\limits_{\gamma \in \Gamma }y_{[s]}(\gamma )d_{a\gamma
a^{-1}\sigma ^{-1}} \\
&=&a_{[p-1]}\sum\limits_{\gamma \in \Gamma }y_{[p-1]}(\gamma \tau
_{j})d_{a\gamma a^{-1}\sigma ^{-1}}\\
& &\mbox{ changing variables $\gamma' = \gamma \tau_j$ and noting that $a\gamma'\tau_j^{-1}a^{-1}\sigma^{-1} = a\gamma'a^{-1} (\sigma \tau_{aj})^{-1}$}\\
&=&y_{[p-1]}(\sigma \tau _{aj})  = \sum\limits_{s=0}^{p-1}\Q_{[p-1,s]}(aj)y_{[s]}(\sigma ).
\end{eqnarray*}

Then by Lemma \ref{lem} item \ref{5} we have that
\[
a_{[s]}\sum\limits_{\gamma \in \Gamma }y_{[s]}(\gamma )d_{a\gamma
a^{-1}\sigma ^{-1}}=y_{[s]}(\sigma ),
\]
for $0\leq s<p$ and $\sigma\in\Gamma$, so statement (b) holds.

$c)\Rightarrow b)$ \\
By hypothesis $y_{[s]}(\gamma _{i})=a_{[s]}\sum\limits_{\gamma \in
\Gamma }y_{[s]}(\gamma )d_{a\gamma a^{-1}\gamma _{i}^{-1}}$ for
$0\leq s<p$, $i=1, \dots,m$ and each digit $\gamma _{i}=(l_{i}, b_{i}).$\\

Let $\sigma \in \Gamma$ then there exists unique $i=1, \dots, m$ and
$\lambda \in \Gamma $, such that $\sigma =\gamma_{i}a\lambda a^{-1}.$ Then $y_{[s]}(\sigma )=y_{[s]}(\gamma
_{i}a\lambda a^{-1})$ and by hypothesis and Lemmas \ref{cor:4.4} and \ref{lem} item \ref{i4}, we have that
\begin{eqnarray*}
y_{[s]}(\gamma _{i}a\lambda a^{-1})&= &\sum\limits_{u=0}^{s}a_{[s]}\Q_{[s,u]}(\lambda)a_{[u]}^{-1}y_{[u]}(\gamma _{i}) \\
&=&\sum\limits_{u=0}^{s}a_{[s]}\Q_{[s,u]}(\lambda)a_{[u]}^{-1}a_{[u]}%
\sum\limits_{\gamma \in \Gamma }y_{[u]}(\gamma )d_{a\gamma
a^{-1}\gamma_{i}^{-1}}  \nonumber \\
&=&a_{[s]}\sum\limits_{\gamma \in \Gamma }\left(
\sum\limits_{u=0}^{s}\Q_{[s,u]}(\lambda)y_{[u]}(\gamma
)\right) d_{a\gamma a^{-1}\gamma _{i}^{-1}}\\
&=&a_{[s]}\sum\limits_{\gamma' \in \Gamma }y_{[s]}(\gamma'
)d_{a\gamma' a^{-1}\sigma ^{-1}} \mbox{ where we again, set }\gamma'=\gamma\lambda.
\end{eqnarray*}
In the last equality we used Lemma~\ref{cor:4.4} that
$\sum\limits_{u=0}^{s}\Q_{[s,u]}(\lambda)y_{[u]}(\gamma
)=y_{[s]}(\gamma \lambda ).$

Therefore
\[
y_{[s]}(\sigma )=a_{[s]}\sum\limits_{\gamma \in \Gamma
}y_{[s]}(\gamma )d_{a\gamma a^{-1}\sigma ^{-1}}.
\]

$c)\Rightarrow d).$ \\
Assume that (c) holds, i.e. $y_{[s]}(\gamma_i)=a_{[s]}\sum_{\gamma\in\Gamma}y_{[s]}(\gamma)d_{a\gamma a^{-1}\gamma_i^{-1}},$ for $0\leq s<p$ and $i=1, \dots, m$. Then
\begin{eqnarray*}
v_{[s]} &=&y_{[s]}(Id)=y_{[s]}(\gamma _{i}\gamma
_{i}^{-1})=\sum\limits_{t=0}^{s}\Q_{[s,t]}(\gamma^{-1}_{i})y_{[t]}(\gamma _{i})  \nonumber \\
&=&\sum\limits_{t=0}^{s}\Q_{[s,t]}(\gamma^{-1}_{i})a_{[t]}\sum%
\limits_{\gamma \in \Gamma }y_{[t]}(\gamma )d_{a\gamma
a^{-1}\gamma
_{i}^{-1}} \\
&=&\sum\limits_{\gamma \in \Gamma
}\sum\limits_{t=0}^{s}\sum%
\limits_{u=0}^{t}\Q_{[s,t]}(\gamma^{-1}_{i})\Q_{[t,u]}(a\gamma a^{-1})a_{[u]}v_{[u]}d_{a\gamma
a^{-1}\gamma _{i}^{-1}}
\nonumber \\
&=&\sum\limits_{\gamma \in \Gamma
}\sum\limits_{u=0}^{s}\sum%
\limits_{t=u}^{s}\Q_{[s,t]}(\gamma^{-1}_{i})\Q_{[t,u]}(a\gamma a^{-1})a_{[u]}v_{[u]}d_{a\gamma
a^{-1}\gamma _{i}^{-1}} \\
&=&\sum\limits_{\sigma \in \Gamma
_{i}}\sum\limits_{u=0}^{s}\Q_{[s,u]}(\sigma^{-1})a_{[u]}v_{[u]}d_{\sigma^{-1}},
\end{eqnarray*}
where the last equality is obtained taking $\sigma =\gamma _{i}a\gamma ^{-1}a^{1}$ and therefore $\sigma\in \Gamma _{i}$.

$d)\Rightarrow c).$\\
 Assume now that (d) holds. Then
\begin{eqnarray*}
y_{[s]}(\gamma _{i})
&=&\sum\limits_{t=0}^{s}\Q_{[s,t]}(\gamma_{i})v_{[t]} \\
&=&\sum\limits_{\gamma \in \Gamma
}\sum\limits_{u=0}^{s}\sum%
\limits_{t=u}^{s}\Q_{[s,t]}(\gamma_{i})\Q_{[t,u]}(a\gamma a^{-1}\gamma_i^{-1})a_{[u]}v_{[u]}d_{a\gamma a^{-1}\gamma _{i}^{-1}}\\
&=&\sum\limits_{\gamma \in \Gamma
}\sum%
\limits_{u=0}^{s}\Q_{[s,u]}(a\gamma  a^{-1})a_{[u]}v_{[u]}d_{a\gamma a^{-1}\gamma _{i}^{-1}}\\
&=&\sum\limits_{\gamma \in \Gamma
}a_{[s]}\left(\sum\limits_{u=0}^{s}\Q_{[s,u]}(\gamma)v_{[u]}\right)d_{a\gamma a^{-1}\gamma _{i}^{-1}}\\
&=&a_{[s]}\sum\limits_{\gamma \in \Gamma }y_{[s]}(\gamma)d_{a\gamma a^{-1}\gamma _{i}^{-1}}.
\end{eqnarray*}
\end{proof}

As in the translation case  the last theorem enables us to obtain a much nicer accuracy condition for $f$.
\begin{theorem}\label{teoaccu}
Let $(\Gamma,G,\Lambda)$ be a splitting crystal triple, $a\in \R^{d\times d}$ a $\Gamma-$admissible matrix, $m=|\det(a)|$ and let $\Lambda_1,...,\Lambda_{m}$ be the (left) cosets of $\Lambda/a\Lambda$. Let $f:\R^d\rightarrow\C$ be a $\Gamma-$refinable function. If the coefficients $d_\gamma$ of the refinement equation (\ref{ec_ref_single}) satisfy:
\begin{enumerate}
\item[i)] $\sum_{\gamma\in\Gamma} d_{\gamma^{-1}}=m$,
\item[ii)] For each $g\in G$ and $|\alpha|<p$
\begin{equation}\label{condicion}
\sum_{\ell\in\Lambda_1}(-g(\ell))^{\alpha}d_{(\ell,g)^{-1}}=\dots=\sum_{\ell\in\Lambda_m}(-g(\ell))^{\alpha}d_{(\ell,g)^{-1}}=\beta_{(\alpha, g)},
\end{equation}
\item[iii)]$1$ is not an eigenvalue of the matrix $\sum_{g\in G}\beta_{(0,g)}g^{-1}_{[s]}a_{[s]}$ for each $0\leq s<p$,
\end{enumerate}
then $f$ has accuracy $p$.
\end{theorem}
These conditions should be compared to Theorem 3.7 in \cite{CHM1}.
\begin{proof}

Note first that  the coefficients $d_\gamma$ are scalars, and hence commute with any matrix or vector.

Is not very difficult to show that if $\{\gamma_1,\dots,\gamma_{m}\}$ is a full set of digits of the left cosets of $\Lambda/a\Lambda$ it is also for the left cosets of $\Gamma/a\Gamma a^{-1}$.	

We define the matrices
$$M_{[s,t]}=\sum_{\gamma\in\Gamma_i}\Q_{[s,t]}(\gamma^{-1})d_{\gamma^{-1}}.$$
Note that  for $\gamma=(\ell,g)$, $\Q_{[s,t]}(\gamma^{-1})=g_{[s]}Q_{[s,t]}(-g(\ell))$, and therefore
$$M_{[s,t]}=\sum_{g\in G}\sum_{\ell\in\Lambda_i} g_{[s]}Q_{[s,t]}(-g(\ell))d_{(\ell,g)^{-1}},$$
and by Lemma~\ref{lem} $\Q_{[s,s]}(\gamma^{-1})=g_{[s]}$.

Since the coefficients $d_{\gamma^{-1}}$ satisfy (\ref{condicion}), the sum
$$\sum_{\ell\in\Lambda_i}g_{[s]}Q_{[s,t]}(-g(\ell))d_{(\ell,g)^{-1}},$$
is independent of $i$. Moreover, as $1$ is not an eigenvalue of $M_{[s,s]}a_{[s]}$,
$\left(I-M_{[s,s]}a_{[s]}\right)$ is invertible.

We shall define scalars $v_\alpha\in\C$ so that $v_{[s]}$ satisfies condition d) of Theorem~\ref{thm}.

Define $v_0=1$. It is not difficult to prove that, $\sum_{\gamma\in\Gamma_i}d_{\gamma^{-1}} =1,$ so $v_{[0]}=[v_0]=1$ satisfies  condition d).

Therefore, if we define the vectors $v_{[s]}$ recursively as
\begin{equation}
v_{[s]}:=\left(I-M_{[s,s]}a_{[s]}\right)^{-1}\sum_{t=0}^{s-1}M_{[s,t]}a_{[t]}v_{[t]},\label{esc}
\end{equation}
they will satisfy condition d) of Theorem~\ref{thm}. To see this, first rewrite (\ref{esc}) as
$$v_{[s]}=M_{[s,s]}a_{[s]}v_{[s]}+\sum_{t=0}^{s-1}M_{[s,t]}a_{[t]}v_{[t]}.$$
Now for $0\leq s<p$ and $i=1, \dots, m,$ let us  compute
\begin{eqnarray*}
\lefteqn{\sum_{\gamma\in\Gamma_i}\sum_{t=0}^s\Q_{[s,t]}(\gamma^{-1})a_{[t]}v_{[t]}d_{\gamma^{-1}}} \\
&=&\sum_{\gamma\in\Gamma_i}\Q_{[s,s]}(\gamma^{-1})a_{[s]}v_{[s]}d_{\gamma^{-1}}+\sum_{t=0}^{s-1}\sum_{\gamma\in\Gamma_i}\Q_{[s,t]}(\gamma^{-1})a_{[t]}v_{[t]}d_{\gamma^{-1}}\\
&=&\sum_{\gamma\in\Gamma_i}g_{[s]}a_{[s]}v_{[s]}d_{\gamma^{-1}}+\sum_{t=0}^{s-1}\left[\sum_{\gamma\in\Gamma_i}\Q_{[s,t]}(\gamma^{-1})d_{\gamma^{-1}}\right]a_{[t]}v_{[t]}\\
&=&M_{[s,s]}a_{[s]}v_{[s]}+\sum_{t=0}^{s-1}M_{[s,t]}a_{[t]}v_{[t]}\\
&=& v_{[s]} \quad \text{where this equality is true by the hypothesis that $\sum d_{\gamma^{-1}} = m$.}
\end{eqnarray*}

Therefore, by Theorem~\ref{teo-3} $f$ has accuracy $p.$
\end{proof}

\subsection{Special vector functions.}

In this section we apply Theorem \ref{teoaccu} to obtain accuracy conditions for a special case of vector (lattice)-refinable functions.

Given $(\Gamma,G,\Lambda)$ a splitting crystal triple, with the point group $G=\{g_1=Id,...,g_r\}$. In \cite{MQ16} the authors show that if we associate to a scalar function $f:\R^d \longrightarrow \C$ the vector valued function $F:\R^d \longrightarrow \C^r$, $F=(f\circ g_1^{-1},...,f\circ g_r^{-1})$, then these two functions have properties in common.

The following definition is important for our purpose.

\begin{definition}
Let $(\Gamma,G,\Lambda)$ be a splitting crystal triple and $G=\{g_1, g_2, \dots, g_r\}$. Let $a$ be a $\Gamma-$admissible matrix and $\{c_k\}_{k\in\Lambda}$, with $c_k\in \mathbb{C}^{r\times r}$. We will say that the matrices $c_k$ have \emph{$(\Gamma,a)-$symmetry}, if
$$c_{i,j}^k=c_{1,\rho_i(j)}^{g^{-1}_{h_i}(k)}\mbox{ for all $i,j=1,...,r$ and $k\in\Lambda.$}$$
where $h_i$ and $\rho_i$ are permutations of $\{1, \dots, r\}$ such that
$$g_{h_i}=ag_ia^{-1}\mbox{ and } g_{\rho_i(j)}=g^{-1}_{h_i}\circ g_j\mbox{ for each $i,j=1, \dots,r.$}$$
\end{definition}

In \cite{MQ16} it is shown that, under some (mild) conditions, $f$ is $\Gamma-$refinable if and only if $F$ is $\Lambda-$refinable. Precisely, they prove the following theorem.

\begin{theorem}\label{relacion}
 Let $(\Gamma,G,\Lambda)$ be a splitting crystal triple, $G=\{g_1=Id,...,g_r\}$, $a$ a $\Gamma-$admissible matrix and $m=|\det a|$.  We consider the sequence $\{c_\gamma\}_{\gamma\in\Gamma}\subset\C$ and $\{\widetilde{c}_k\}_{k\in\Lambda}\in \C^{r\times r}$, where the matrices $\widetilde{c}_k$ are related to the scalars $c_\gamma$ by the equality
$$\widetilde{c}_k=(c_{i,j}^{k})_{i,j=1,...,r}=\left(c_{(g^{-1}_{h_i}\circ g_j,g^{-1}_j(k))}\right)_{i,j=1, \dots, r}.$$
Then
\begin{enumerate}
  \item If $f:\R^d\rightarrow\C$ is $\Gamma-$refinable, then the function $F=(f,f\circ g_2^{-1},...,f\circ g_r^{-1})$ is $\Lambda-$refinable and the coefficients of the $\Lambda$-refinement equation have $(\Gamma, a)-$symmetry.
  \item If $\sum_{\gamma\in\Gamma}|c_\gamma|^2<m$ and $F=(f_1, \dots, f_r)\in L^2(\R^d,\C^r)$ is the solution of the refinement equation associated to the matrices $\{\widetilde{c}_k\}_{k\in\Lambda}$, then $F=(f,f\circ g_2^{-1},...,f\circ g_r^{-1})$ and the function $f=f_1$ is the solution of the $\Gamma-$refinement equation associated  to the scalars $\{c_\gamma\}_{\gamma\in\Gamma}$, i.e., $f$ is solution of
$$f(x)=\sum_{\gamma\in\Gamma}c_\gamma f(\gamma ax)\;\; \mbox{ a.e. }x\in\R^d.$$
\end{enumerate}
\end{theorem}

From  Theorem \ref{relacion} together with Theorem \ref{teoaccu}, we present a much simpler condition for characterizing the accuracy of some special functions $F:\R^d\rightarrow\C^r$.

\begin{theorem} \label{simple}
Let $(\Gamma,G,\Lambda)$ be a splitting crystal triple and $G=\{g_1=Id, g_2, \dots, g_r\}$. Let $a$ be a $\Gamma-$admissible matrix and $m=|\det a|$. Let $F:\R^d\rightarrow\C^r$ be a function such that $F=(f,f\circ g_2^{-1}, \dots, f\circ g_r^{-1})$, is $\Lambda-$refinable and the coefficients  $\widetilde{c}_k$ of the $\Lambda$-refinement equation have $(\Gamma, a)-$symmetry. Consider the scalars $c_\gamma=c_{(l,g_i)}=\widetilde{c}^{g_i(l)}_{1,i}=(\widetilde{c}_{g_i(l)})_{1,i}$, generated by the matrices $\widetilde{c}_k$. If the sequence $\{c_\gamma\}_{\gamma\in\Gamma}$ satisfies the hypothesis of Theorem~\ref{teoaccu} and $\sum_{\gamma\in\Gamma}|c_\gamma|^2<m$, then $F$ has accuracy $p$.
\end{theorem}

Compare this to the conditions of Theorem 3.4 in \cite{CHM1}. The conditions of the previous Theorem are clearly much easier to check!

\begin{proof}
Without loss of generality, we assume that $g_1=Id$. By Theorem~\ref{relacion} $f=f_1$ is a $\Gamma-$refinable function, and  $\{c_\gamma\}_{\gamma\in\Gamma}$ are the coefficients of the $\Gamma-$refinement equation. Further $\{c_\gamma\}_{\gamma\in\Gamma}$ satisfy the hypothesis of Theorem~\ref{teoaccu}, therefore the function $f$ has accuracy $p$.

To show that $F$ has accuracy $p$ let $P(x)$ a polynomial of degree less than $p$. Then
\begin{equation}
P(x) = \sum_{\gamma\in\Gamma}c_{\gamma}f(\gamma(x)) = \sum_{k\in \Lambda}\sum_{i=1}^n c_{(k,g_i)}f(g_i^{-1}(x-k)) = \sum_{k\in \Lambda}C_k F(x-k).\label{12}
\end{equation}
Then $F$ reproduces the same polynomials than $f$. Therefore $F$ has accuracy $p$.
\end{proof}

From equality (\ref{12}) we have in fact the following result.

\begin{cor}\label{equiv}
Let $(\Gamma,G,\Lambda)$ be a splitting crystal triple and $G=\{g_1={\rm Id}, g_2, \dots, g_r\}$. Let $a$ be a $\Gamma-$admissible matrix and $m=|\det A|$. Let $f:\R^d\rightarrow\C$, $f\in L^2(\R^d)$ and $F:\R^d\rightarrow\C^r$ be defined by $F=(f,f\circ g_2^{-1},...,f\circ g_r^{-1})$. Then $f$ has accuracy $p$ if and only if $F$ has accuracy $p$.
\end{cor}

\subsection{Accuracy and Order of Approximation.} \label{approx-accuracy}

The notion of accuracy has been studied before in the context of approximation theory and can be related to properties of the space $\mathcal S(f)$ (see equation (\ref{SF})). In this section we will discuss the connection between accuracy and order of approximation for crystal-invariant spaces. We will state our results for $L^2(\R^d)$, but it can also be formulated for $L^q(\R^d), q \geq 1$.

Let $\mathcal{S} :=\mathcal{S}(F)\cap L^2(\R^d)$, and set $\mathcal{S}^h=\{g(x/h):g\in\mathcal{S}\}$. Let $W_n^2(\R^d)$ denote the Sobolev space consisting of all functions whose weak derivatives up to order $n$ all lie in $L^2(\R^d)$.

\begin{definition}
  We say that $\mathcal{S}(F)$ provides $L^2-${\em approximation order} $n$ if for each $g\in W^2_n(\R^d)$ there exists a constant $c_g$ independent of $h$ such that
  \[\mbox{for all } h>0; \inf_{k\in\mathcal{S}^h}\|g-k\|_2\leq c_g h^n.\]

  We say that $\mathcal{S}(F)$ provides $L^2-${\em density order} $n$ if for each $g\in W^2_n(\R^d)$
  \[\lim_{h\rightarrow0}(\inf_{k\in\mathcal{S}^h}\|g-k\|_2)/h^r=0.\]
\end{definition}

Let us recall the general Poisson formula for a function $f\in L^1(\R^n)$ and a lattice $\Ll$. Consider  $f$ with compact support, $\Ll$ a lattice and $\Ll^*$ its dual. We then have
	$$\sum_{k\in L}f(x+k)=|L|^{-1}\sum_{l\in L^*}\widehat{f}(l)e^{-2\pi i\langle l,x\rangle},$$
where $L$ is a fundamental domain of $\Ll$.


Now, we recall the Strang-Fix conditions for a single function $f:\R^d\rightarrow\C$ and a vector function $F:\R^d\rightarrow\C^r$, and generalize them to the crystal setting.

\begin{definition}
  Let $f:\R^d\rightarrow\C$ be a compactly supported function in $L^2(\R^d)$, $\Ll$ a lattice, $\Ll^*$ its dual and $\alpha$ a multi-index, we say that $f$ satisfies the Strang-Fix conditions of order $n$ if
  \begin{equation}\label{Strang}
    \widehat{f}(0)\neq0 \mbox{ and }D^\alpha\widehat{f}(\ell)=0,\mbox{ for all }\ell\in\Ll^*,0\leq |\alpha|\leq n-1.
  \end{equation}
  Let $F=(f_1, \dots, f_r)^t:\R^d\rightarrow\C^r$ be a vector of compactly supported functions, we say that $F$ satisfies the Strang-Fix conditions of order $n$ if there exists a function $g$ which is a finite linear combination of lattice translates of $f_1, \dots, f_r$, i.e.,
  \[g(x)=\sum_{i=1}^r\sum_{k\in M}c_{k,i}f_i(x-k),\]
  and which satisfies the Strang-Fix conditions (\ref{Strang}), where $M$ is a finite subset of $\Ll$.
  We say that $F$ satisfies the {\em crystal Strang-Fix conditions}, if $F=(f\circ{g_1^{-1}}, \dots, f\circ{g_r^{-1}})^t:\R^d\rightarrow\C^r$, with $g_i \in $ the point group of $\Gamma$,  and $F$ satisfies the Strang-Fix conditions for the lattice $\Ll=\Lambda$ associated to $\Gamma$.
\end{definition}

Before stating the main theorem of this section, we show the relation between accuracy and Strang-Fix condition for a function $f$, in the context of translations.

\begin{theorem}
		 Let $f:\R^n\rightarrow\mathbb{C}$ a function with compact support such that $x^\alpha f(x)\in L^1(\R^n)$, for all multi-indices $\alpha$ with $|\alpha|\leq p-1$, then the following  are equivalent:
		 \begin{enumerate}
		 	\item $f$ satisfies the Strang-Fix conditions of orden $p$
		 	\item For each multi-index $\alpha$ with $|\alpha|\leq p-1$, $\displaystyle\sum_{k\in L} k^\alpha f(x-k)$ is a polynomial of degree $|\alpha|$, moreover the coefficient of $x^\alpha$ is non-zero.
		 \end{enumerate}
	\end{theorem}

\begin{proof}
Since $x^{\alpha} f(x)\in L^1(\R^n)$ we have that $D^\alpha\widehat{f}(x)=\frac{\partial^{|\alpha|}\widehat{f}}{\partial x^{\alpha}}(x)$ exists for each $x\in\R^d$.
		We consider the function $\psi(y)=y^\alpha f(x-y)$ where $x\in \R^n$ is fixed. Its Fourier transform is
		\begin{eqnarray*}
			\widehat{\psi}(\xi)&=&\int_{\R^n}f(x-y)y_1^{\alpha_1}...y_n^{\alpha_n} e^{-2\pi i y_1\xi_1}...e^{-2\pi i y_n\xi_n}dy\nonumber\\
			&=&\int_{\R^n}f(x-y)\frac{1}{(-2\pi i)^{|\alpha|}}\frac{\partial^{|\alpha|}}{\partial(\xi^\alpha)}(e^{-2\pi i\langle y,\xi\rangle})dy\nonumber\\
			&=&\frac{1}{(-2\pi i)^{|\alpha|}}\frac{\partial^{|\alpha|}}{\partial(\xi^\alpha)}\left(\int_{\R^n}f(x-y)e^{-2\pi i\langle y,\xi\rangle}dy\right)\nonumber\\
			&=& \frac{1}{(-2\pi i)^{|\alpha|}}\frac{\partial^{|\alpha|}}{\partial(\xi^\alpha)}
			\left(e^{-2\pi i\langle y,\xi\rangle}\widehat{f}(-\xi)\right).\nonumber
		\end{eqnarray*}
		
		Then,  the by the Poisson formula for $\psi$ for each $x$ we have		
		\begin{eqnarray}\label{pois}
		\lefteqn{\sum_{k\in L}k^\alpha f(x-k) = } \nonumber\\
		&=&\frac{|P|^{-1}}{(-2\pi i)^{|\alpha|}}\sum_{l\in L^*}D^\alpha \left.(e^{-2\pi i\langle x,\xi\rangle}\widehat{f}(-\xi))\right|_{\xi=l} \nonumber\\
		&=&\frac{|P|^{-1}}{(-2\pi i)^{|\alpha|}}\sum_{l\in L^*}\left.\left(\sum_{\beta\leq\alpha}\left(\begin{array}{c}
		\alpha \\
		\beta \\
		\end{array}\right)D^\beta(e^{-2\pi i\langle x,\xi\rangle})D^{\alpha-\beta}\widehat{f}(-\xi)\right)\right|_{\xi=l}.
		\end{eqnarray}
		
		By hypothesis, in this last sum the only non-zero terms are those corresponding to $l=0$. Therefore
		$$\sum_{k\in L}k^\alpha f(x-k)=\frac{|P|^{-1}}{(-2\pi i)^{|\alpha|}}\sum_{\beta\leq\alpha}(-2\pi i)^{|\beta}x^\beta(-1)^{|\alpha|-|\beta|}D^{\alpha-\beta}\widehat{f}(0),$$
		which is a polynomial in $x$ of degree $|\alpha|$ because when $\beta=\alpha$ the coefficient is $|L|^{-1}\widehat{f}(0)\neq0$.
		
		Now we assume that 2. holds. Taking $\alpha=0$ in (\ref{pois}) we have that the Fourier series in $L^2(\R^n/\Ll)$ of the constant function $\displaystyle\sum_{k\in L}f(x-k)$ is $|L|^{-1}\sum_{l\in L^*}e^{-2\pi i\langle x,l\rangle}\widehat{f}(-l)$, and therefore $\widehat{f}(0)\neq0$ and $\widehat{f}(l)=0$ for all $l\in \Ll^*$ and $l\neq0$.
		
		We now consider the multi-index $\alpha=(1,0,...,0)$. Then
		\begin{eqnarray*}
		\frac{|P|^{-1}}{(-2\pi i)}\left(\sum_{l\in L^*-{0}}\left(e^{-2\pi i\langle x,l\rangle}(-1)\frac{\partial\widehat{f}}{\partial\xi_1}(-l)+(-2\pi i)x_1 e^{-2\pi i\langle x,l\rangle}\widehat{f}(-l)\right)+\right.\\
		\left.+(-1)\frac{\partial\widehat{f}}{\partial\xi_1}(0)+(-2\pi i)x_1\widehat{f}(0)\right),
		\end{eqnarray*}
		is a polynomial whose main coefficient  is $cx$ with $c\neq0$ and $\widehat{f}(l)=0$ if $l\neq0$. Therefore  $\frac{\partial\widehat{f}}{\partial\xi_1}(l)=0$ for $l\neq0$ and $\widehat{f}(0)\neq0.$ Repeating this argument for $\alpha=e_i$ where $e_i\in\R^d$ is the vector with entries $0$ in the place $j\neq i$ and $1$ in the place $i$, we obtain that $f$ satisfies the Strang-Fix contions of order $p$.
	\end{proof}

The following result shows that if $f$ is a crystal-refinable function with compact support and $\Gamma-$independent translates, then order of approximation, density order, Strang-Fix conditions and accuracy are equivalent.
\begin{remark}\label{obsacu}
  When $\Gamma$ consists only of translations, i.e. $G=\{g_1=Id\}$, this theorem was proved in \cite{Jia}.
\end{remark}

\begin{theorem} \label{teo-accu-approx}
  Let $(\Gamma, G, \Lambda)$ be a splitting crystal triple, $G=\{g_1=Id, \dots, g_r\}$ and $f\in L^2(\R^d)$ be a function with compact support and $\Gamma-$independent translates. We consider the function $F=(f, f\circ g_2^{-1},...,f\circ g_r^{-1})$. The following statements are equivalent:
  \begin{enumerate}
    \item $f$ has accuracy $p$.
    \item $\mathcal{S}(F)$ provides $L^2$-density order $p-1$.
    \item $\mathcal{S}(F)$ provides $L^2$-approximation order $p$.
    \item $F$ satisfies the Strang Fix conditions of order $p$
  \end{enumerate}
\end{theorem}
\begin{proof}
  If $f$ has $\Gamma-$independent translates it is immediate that the vector-function $F$ has independent translates with respect to the lattice $\Lambda$ associated to $\Gamma$.
  By Corollary \ref{equiv} we know that $f$ has accuracy $p$ if and only if $F$ has accuracy $p$.
  Therefore, by Remark \ref{obsacu}, it is equivalent for $f$ to have accuracy $p$, that $\mathcal{S}(F)$ provides $L^2$-approximation order $p$, which in turn is equivalent to $\mathcal{S}(F)$ providing $L^2$-density order $p-1$, and this is equivalent to $F$ satisfying the Strang-Fix conditions of order $p$.
  Therefore $f$ has accuracy $p$ if and only if $\mathcal{S}(F)$ provides $L^2$-approximation order $p$ if and only if $\mathcal{S}(F)$ provides $L^2$-density order $p-1$, if and only if $F$ satisfies the Strang-Fix conditions of order $p$.
\end{proof}

\section{Statements of the Theorems for the multi function case}

The main theorems of this paper, can be extended, using the techniques introduced in \cite{CHM1} for the case of vector-valued functions. We will state the theorems in full generality, but leave the proofs for the interested reader.

We will say that a vector valued function: $\phi(x) : \R^d \longrightarrow \C^\ell, \phi(x) := (\varphi_1(x) \dots, \varphi_\ell(x))^T$ is $\Gamma-$refinable, if it satisfies the refinement equation:
\begin{equation}\label{ec_ref}
  \phi(x)=\sum_{\gamma\in\Gamma'}d_\gamma \phi(\gamma^{-1}ax),
\end{equation}
for some finite $\Gamma'\subset\Gamma$, and matrices $d_{\gamma} \in \C^{\ell\times \ell}$. These matrices $d_{\gamma}$ are called \emph{coefficients of the refinement equation}.

Given a collection
$$\{v_\al=(v_{\al,1},...,v_{\al,\ell})\in\C^{1\times \ell}:0\leq|\al|<p\},$$
of row vectors of length $\ell$, we group the $v_\al$ by degree to form $d_s\times1$ column vectors $v_{[s]}$ with block entries that are the $1\times \ell$ row vectors $v_\al$. Specifically, we set

$$  v_{[s]}=[v_\al]_{|\al|=s}, \ \ 0\leq s<p.$$

Note that, when $\al=0$ then $v_{[0]}=[v_0]=v_0$.

We define the matrices $y_{[s]}(\gamma)$ as before by
\[
y_{[s]}(\gamma)=\sum_{t=0}^s\Q_{[s,t]}(\gamma)v_{[t]},\]
but noting that now $v_{[t]}$ are matrices of size $d_t\times \ell$.

\begin{theorem}
Assume that $f:\mathbb{R}^{d}\rightarrow\mathbb{C}^{\ell}$ is integrable, compactly supported and satisfies
the refinement equation (\ref{ec_ref}).  Consider the following statements

\begin{itemize}
\item[I)] $f$ has accuracy $p$.

\item[II)] There exist a collection of row vectors $\{v_{\alpha }\in\mathbb{C}^{1\times \ell}:0\leq \left\vert \alpha \right\vert <p\}$
such that

\begin{itemize}
\item[(i)] $v_{0}\hat{f}(0)\neq 0$ and

\item[(ii)] $Y_{[s]}=a_{[s]}Y_{[s]}L$ for $0\leq s<p$
where $Y_{[s]}=(y_{[s]}(\gamma ))_{\gamma \in \Gamma }$ as in (\ref{y-s}) and (\ref{Y-s}).
\end{itemize}
\end{itemize}

Then we have the following:

\begin{itemize}
\item[a)] If the translates of $f$ along $\Gamma$ are independent, then \emph{(I)}
implies \emph{(II)}.

\item[b)] \emph{(II)} implies \emph{(I)}. In this case, if we
scale all the vectors $v_{\alpha }$ by
$C=(v_{0}\hat{f}(0))^{-1}\left\vert P\right\vert $ then
$$
X_{[s]}(x)=\sum\limits_{\gamma \in \Gamma }y_{[s]}(\gamma
)f(\gamma (x))=Y_{[s]}F(x)\mbox{, }0\leq s<p.
$$
\end{itemize}
\end{theorem}

As for the single function case, this theorem can be simplified so that, if one wants to check for accuracy $p$, one does not need to check {\bf all} conditions $0\leq s < p$, but it is enough to check them for $s=p-1$.

\begin{theorem} Assume that $f:\mathbb{R}^{d}\rightarrow\mathbb{C}^{m}$ is integrable, compactly supported and satisfies
the refinement equation (\ref{ec_ref}).
Let $\de=\left\vert \det a\right\vert ,$ and let $\gamma_{1}, \dots,\gamma _{\de}$
 be a full set of digits of the left cosets of $\Gamma $. Here,  the left cosets $\Gamma_i$ are $\Gamma _{i}=\gamma _{i}a\Gamma a^{-1}, i = 1, \dots, \de.$

Given a collection $\{v_{\alpha }\in\mathbb{C}^{1\times r}:0\leq \left\vert \alpha \right\vert <p\}$ of row vectors, let\\
$y_{[s]}(\gamma)=\sum\limits_{t=0}^{s}\Q_{[s,t]}(\gamma)v_{[t]}$ and  $Y_{[s]}=(y_{[s]}(\gamma ))_{\gamma \in \Gamma }.$

If $v_{0}\neq 0,$ then the following statements are equivalent:

\begin{itemize}
\item[a)] $Y_{[p-1]}=a_{[p-1]}Y_{[p-1]}L$.
Equivalently, \\$%
y_{[p-1]}(\sigma )=a_{[p-1]}\sum\limits_{\gamma \in \Gamma
}y_{[p-1]}(\gamma )d_{a\gamma a^{-1}\sigma ^{-1}}$ for $\sigma
\in \Gamma. $

\item[b)] $Y_{[s]}=A_{[s]}Y_{[s]}L$ for $0\leq s<p$. Equivalently, \\
$%
y_{[s]}(\sigma )=A_{[s]}\sum\limits_{\gamma \in \Gamma
}y_{[s]}(\gamma )d_{A\gamma A^{-1}\sigma ^{-1}}$ for $\sigma
\in \Gamma. $ \item[c)] $y_{[s]}(\gamma _{i})=A_{[s]}\sum\limits_{\gamma \in \Gamma
}y_{[s]}(\gamma )d_{A\gamma A^{-1}\gamma _{i}^{-1}}$ for $0\leq s<p$ and $i=1,...,\de.$

\item[d)] $v_{[s]}=\sum\limits_{\gamma \in \Gamma
_{i}}\sum\limits_{t=0}^{s}\Q_{[s,t]}(\gamma^{-1})A_{[t]}v_{[t]}d_{\gamma
^{-1}}$ for $0\leq s<p$ and $i=1, \dots, \de.$
\end{itemize}
\end{theorem}

\section{Acknowledgements}
The authors gratefully acknowledge support from MinCyT, ANPCyT PICT2014-1480 and UBA, UBACyT 20020130100403BA.

\end{document}